\documentclass[11pt]{amsart}
\usepackage{amssymb,amsmath,txfonts,mathrsfs,mathtools,titletoc, tikz, stmaryrd}
\usepackage{color}
\usepackage{graphicx}
\usepackage{float}
\usepackage{appendix}
\usepackage{hyperref}

\usepackage[alphabetic]{amsrefs}
\newtheorem{theorem}{Theorem}[section]
\newtheorem{prop}[theorem]{Proposition}
\newtheorem{lemma}[theorem]{Lemma}
\newtheorem{remark}[theorem]{Remark}
\newtheorem{claim}[theorem]{Claim}

\newtheorem{notation}[theorem]{Notation}
\newtheorem{definition}[theorem]{Definition}

\DeclareMathOperator{\ep}{\epsilon}

\numberwithin{equation}{section}

\def\pf{{\it Proof:}~}

\usepackage{color}
\usepackage{graphicx}
\usepackage{float}
\usepackage{appendix}

\usepackage[alphabetic]{amsrefs}
\numberwithin{equation}{section}

\usepackage[alphabetic]{amsrefs}

\numberwithin{equation}{section}

\def\pf{{\it Proof:}~}

\def\pf{{\it Proof:}~}

\DeclareMathOperator{\p}{\partial}
\DeclareMathOperator{\lan}{\langle}
\DeclareMathOperator{\ran}{\rangle}
\DeclareMathOperator{\e}{\epsilon}

\numberwithin{equation}{section}

\def\pf{{\it Proof:}~}

\begin{document}

\title[The rigidity of eigenfunctions' gradient estimates]{The rigidity of eigenfunctions' gradient estimates}
\author{Guoyi Xu,  Xiaolong Xue}
\address{Guoyi Xu\\Department of Mathematical Sciences\\Tsinghua University, Beijing\\P. R. China, 100084}
\email{guoyixu@tsinghua.edu.cn}
\address{Xiaolong Xue\\ Department of Mathematical Sciences\\Tsinghua University, Beijing\\P. R. China}
\email{xxl19@mails.tsinghua.edu.cn}
\date{\today}
\date{\today}

\begin{abstract}
On compact Riemannian manifolds with non-negative Ricci curvature and smooth (possibly empty),  convex (or mean convex) boundary, if the sharp Li-Yau type gradient estimate of an Neumann (or Dirichlet) eigenfunction holds at some non-critical points of the eigenfunction; we show that the manifold is isometric to the product of one lower dimension manifold and a round circle (or a line segment).
\\[3mm]
Mathematics Subject Classification: {35J15,  58J05, 58J50. }
\end{abstract}
\thanks{G.Xu was partially supported by NSFC 12141103.}

\maketitle

\titlecontents{section}[0em]{}{\hspace{.5em}}{}{\titlerule*[1pc]{.}\contentspage}
\titlecontents{subsection}[1.5em]{}{\hspace{.5em}}{}{\titlerule*[1pc]{.}\contentspage}
\tableofcontents

\section{Introduction}\label{sec introduction}

For manifolds with non-negative Ricci curvature ($Rc\geq 0$), Cheeger-Gromoll's splitting theorem \cite{CG} says: the existence of geodesic line results in the isometric splitting of the manifold. It can be shown that the Busemann function with respect to the geodesic line is a harmonic function, whose gradient has constant norm. Then the isometry is induced by the gradient curve of this harmonic function. 

\begin{remark}\label{rem max of grad of harmonic func is obtained}
{Note the square norm of harmonic functions' gradient is a subharmonic function in manifolds with $Rc\geq 0$.  If the maximum norm of harmonic functions' gradient is obtained at some point $p\in M^n$,  then the strong maximum principle implies that the norm of harmonic functions' gradient is constant, which guarantees that the manifold isometrically splits. 

Therefore from the pure analytic view point, the above isometric splitting result can be viewed as the rigidity part of harmonic functions' gradient estimate.
}
\end{remark}

\begin{figure}[htbp]
        \centering
        \includegraphics[scale=1.1]{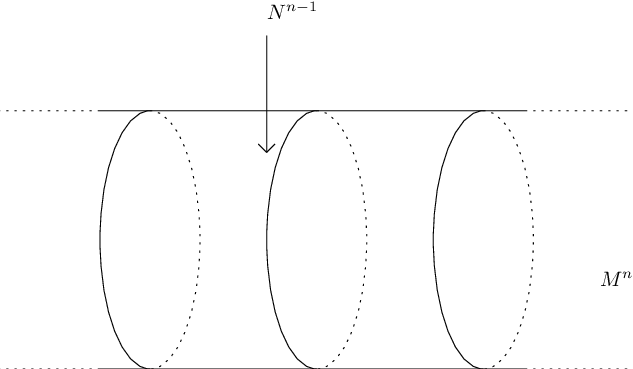}
        \caption{The Figure for Cheeger-Gromoll's splitting theorem with $M^n=N^{n-1}\times \mathbb{R}$}
        \label{The Figure for Cheeger-Gromoll's splitting theorem}
    \end{figure}

For the global positive Green's function, the sharp estimate and the related rigidity are known to Colding \cite{Colding}. For the global positive Green's function $G$ on a complete Riemannian manifold with nonnegative Ricci curvature, it is shown that $|\nabla (G^{\frac{1}{2 -n}})|\leq 1$, and $M^n$ is isometric to $\mathbb{R}^n$ if and only if the equality holds at some point in $M^n$. 

The harmonic functions in Cheeger-Gromoll's splitting theorem and Colding's rigidity are global. The study of local harmonic functions in manifolds with non-negative Ricci curvature, was initiated by Yau's Liouville theorem \cite{Yau} and  Cheng-Yau's gradient estimate for positive harmonic functions \cite{CY}, where the gradient of the logarithm of positive harmonic functions is considered. 

In \cite{Xu-growth},  the sharp Cheng-Yau type gradient estimate in two dimensional manifolds with $Rc\geq 0$ is obtained.  Furthermore,  the corresponding rigidity is established in \cite{HXY} as follows: the realization of maximum norm of positive harmonic function's gradient in geodesic disk implies, the geodesic disk is isometric to an Euclidean disk. However, the sharp Cheng-Yau type gradient estimate is unknown for higher dimension cases so far.



Li \cite{Li} and Li-Yau \cite{LY} applied the gradient estimates of eigenfunctions to study the bound of the eigenvalues on manifolds. The simplest model for eigenfunctions are sine or cosine functions. The Li-Yau type gradient estimate of eigenfunctions (the estimate of $|D\sin^{-1}u|$, where $u$ is one normalized eigenfunction of compact manifold) is firstly  obtained in \cite[Theorem $10$]{LY}. 

Through a delicate refined gradient estimate of eigenfunctions, Zhong-Yang \cite{ZY} firstly obtain the sharp lower bound of the first eigenvalue for compact Riemannian manifolds with $Rc\geq 0$.  Hang-Wang \cite{HW} answered Sakai's question in \cite{Sakai}, they showed that if the sharp lower bound of the $1$st eigenvalue is obtained, then manifold is isometric to a round circle. Zhong-Yang's type gradient estimate of eigenfunction is closely related to the solution of suitable ODE (also see \cite{Kroger} and \cite{AX}). There is also log-concavity property of the first eigenfunctions established by two-points method in \cite{AC} (also see \cite{Ni}). For Riemannian manifold satisfying $Rc\ge -(n-1)$, the sharp gradient estimate of the $1$st eigenfunction and the rigidity of sharp $1$st eigenvalue is studied in \cite{SW} (also see \cite{Li12}). 

Although Li-Yau type gradient estimate of eigenfunctions is not sharp enough to get the sharp lower bound of the first eigenvalue, the equality in the estimate can be achieved for model spaces; which splits a round circle factor isometrically. In this paper, we study the rigidity part with respect to Li-Yau type gradient estimates of eigenfunctions, which is partially motivated by Remark \ref{rem max of grad of harmonic func is obtained}. 

The main result of this paper is as follows, which can be viewed as an isometric splitting theorem for compact manifolds comparing Cheeger-Gromoll's splitting Theorem for complete non-compact manifolds. 
\begin{theorem}\label{thm splitting for eigen est}
{For compact Riemannian manifold $(M^n, g)$ with $Rc\geq 0$ and smooth (possibly empty) boundary . Assume $\Delta u= -\lambda u, \displaystyle \max_{x\in M^n}|u(x)|= 1$, where $\lambda> 0$; and $|\nabla \sin^{-1}u|^2(p)= \lambda$ holds at some point $p$, where $|u(p)|< 1$. Define $\mathcal{N}(u)= \#(M^n- u^{-1}(0))$ as the number of connected components of $M^n- u^{-1}(0)$ (the number of  nodal domains of $u$); then:
\begin{enumerate}
\item[(1)]. If $\partial M^n= \emptyset$, then $M^n$ is isometric to $\mathbb{S}^1(\frac{\mathcal{N}(u)}{2\sqrt{\lambda}})\times N^{n- 1}$, where $\mathbb{S}^1(\frac{\mathcal{N}(u)}{2\sqrt{\lambda}})$ is a round circle with radius $\frac{\mathcal{N}(u)}{2\sqrt{\lambda}}$ and $\mathcal{N}(u)= 2k$ for some $k\in \mathbb{Z}^+$;
\begin{figure}[htbp]
        \centering
        \includegraphics[scale=0.6]{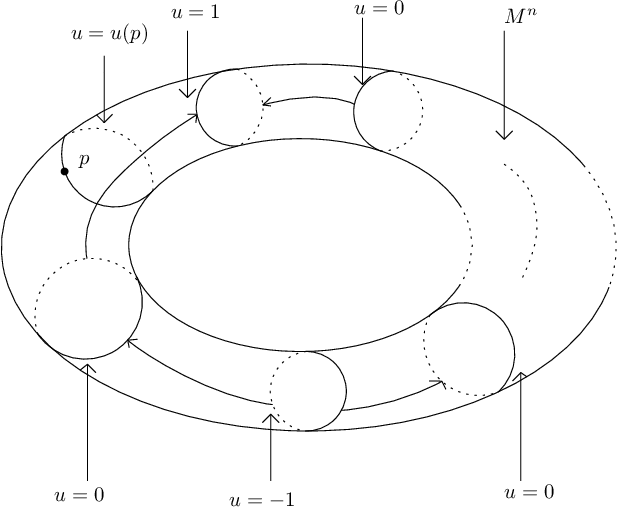}
        \caption{The Figure for Case (1) of Theorem 1.2}
        \label{The Figure for Case (1) of Theorem 1.2}
    \end{figure}
\item[(2)]. If $\partial M^n\neq \emptyset$ is convex and $\frac{\partial u}{\partial \vec{n}}\big|_{\partial M^n}= 0$, then $M^n$ is isometric to $[0, \frac{\mathcal{N}(u)- 1}{\sqrt{\lambda}}\cdot \pi]\times N^{n- 1}$ and $\mathcal{N}(u)\geq 2$;
\begin{figure}[htbp]
        \centering
        \includegraphics[scale=0.7]{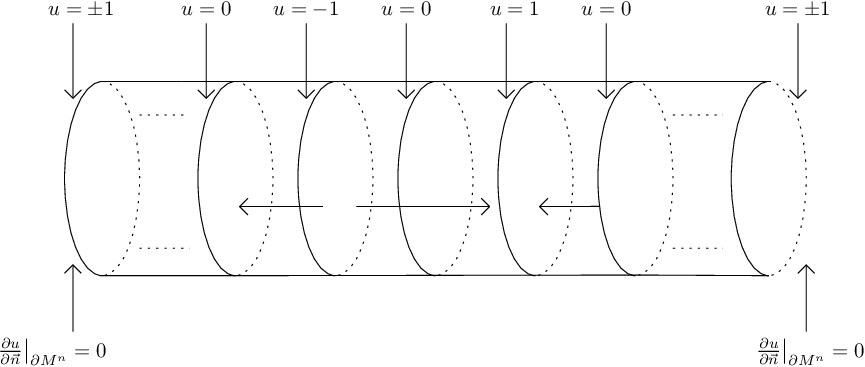}
        \caption{The Figure for Case (2) of Theorem 1.2}
        \label{The Figure for Case (2) of Theorem 1.2}
    \end{figure}
\item[(3)]. If $\partial M^n\neq \emptyset$ is mean convex and $u\big|_{\partial M^n}= 0$; then $M^n$ is isometric to $[0, \frac{\mathcal{N}(u)}{\sqrt{\lambda}}\cdot \pi]\times N^{n- 1}$ and $\mathcal{N}(u)\geq 1$;
\begin{figure}[htbp]
        \centering
        \includegraphics[scale=0.7]{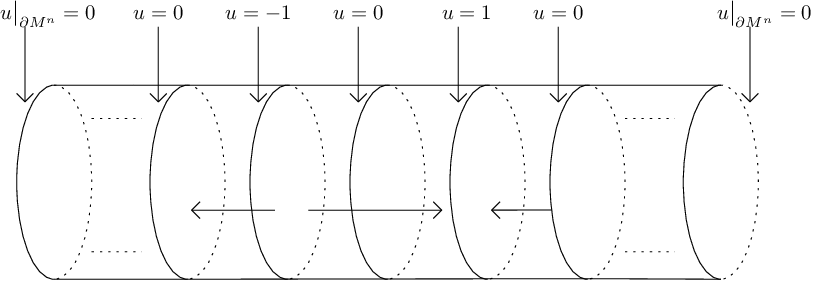}
        \caption{The Figure for Case (3) of Theorem 1.2}
        \label{The Figure for Case (3) of Theorem 1.2}
    \end{figure}
\end{enumerate}
where $N^{n- 1}$ is a compact Riemannian manifold with $Rc\geq 0$ and $\partial N^{n- 1}= \emptyset$.  
}
\end{theorem}

\begin{remark}\label{rem grad of curve}
{The gradient curve of Busemann function is used in \cite{CG72} to show the homeomorphism type of the manifolds. Furthermore the gradient curve of harmonic functions is used to obtain the isometric splitting in \cite{CG}. The proof of the above theorem relies on the gradient curve of $\sin^{-1}u$. However, because of the possible boundary of manifolds and the critical points of eigenfunctions, we firstly obtain a locally splitting result, then extend this splitting to the global manifold. 
}
\end{remark}

The organization of this paper is as follows.  In Section \ref{sec sharp est},  we firstly give an alternative proof of the sharp gradient estimate of $\sin^{-1}u$,  where $u$ is one normalized eigenfunction.  The sharp gradient estimate is known in equivalent form (see \cite{Li12}, \cite{Yang}),  however our method is uniform and consistent in dealing with the corresponding rigidity part later.

In Section \ref{sec local splitting}, we obtain the vanishing Hessian of $\sin^{-1}u$, then using the gradient curve of this function to show the locally splitting in the connected component without critical points of eigenfunctions.  

Finally we show the locally splitting can be extended to the globally splitting in Section \ref{sec splitting extension}. Note the canonical strong maximum principle (see\cite{GT}, \cite{PW}) requires the uniformly elliptic property of the differential operators. Because the elliptic operator related to $\sin^{-1}u$ is degenerate at the critical points of $u$,  we can not apply the canonical strong maximum principle directly.

The key observation is that using the Taylor expansion of the Riemannian metric and eigenfunction at the critical points,  we find a suitable make-up function,  which plays the similar role of the make-up function used in the proof of classical Hopf's Strong Maximum Principle (see \cite[Lemma $3.4$]{GT}).  Therefore,  we obtain the implication of Strong Maximum Principle for this degenerate elliptic operator,  which guarantees the extension of locally splitting.  

\begin{figure}[htbp]
        \centering
        \includegraphics[scale=0.7]{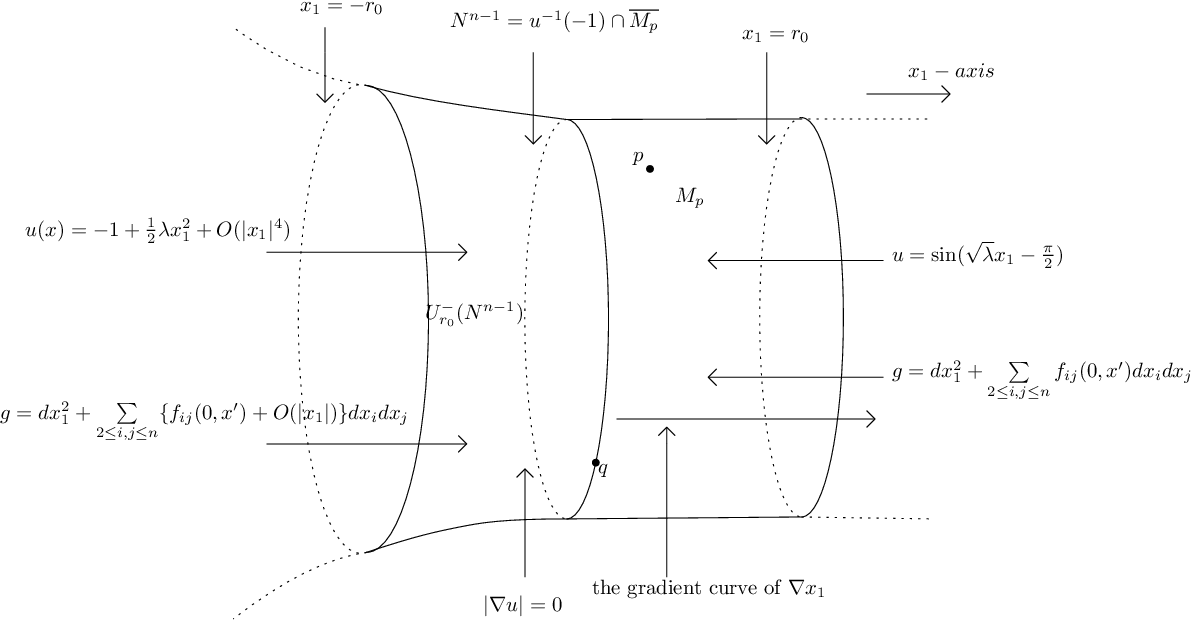}
        \caption{The Figure  of Theorem 4.1 from local splitting to global splitting}
        \label{The Figure  of Theorem 4.1 from local splitting to global splitti}
    \end{figure}

\section{The sharp gradient estimates for eigenfunctions}\label{sec sharp est}

We firstly study the behavior of eigenfunctions' gradient at the boundary, where Neumann boundary condition and smooth convex boundary (respectively Dirichlet boundary condition and smooth, mean convex boundary) are used. 

\begin{lemma}\label{lem boundary computation,dirichlet,1.22}
For compact Riemannian manifold $(M^n, g)$, assume $\Delta u= -\lambda u$ with $|u|\leq 1$, where $\lambda> 0$. If $\partial M^n$ is smoothly convex and $\frac{\partial u}{\partial \vec{n}}\big|_{\partial M^n}= 0$, or if $\partial M^n$ is smooth,  mean convex and $u\big|_{\partial M^n}= 0$; then  
\begin{align}
\frac{\p |\nabla \sin^{-1}u|^2}{\p \vec{n} }(x)\le 0, \quad \quad \quad \quad \forall x\in \partial M^n\ \text{satisfying}\ |u(x)|< 1; \nonumber 
\end{align}
where $\vec{n}(x)$ is the outward unit normal vector on $\p M^n$.
\end{lemma}

\begin{proof}
Let $\{e_1, \cdots,e_{n-1},e_n\}$ be a local orthogonal  frame such that $\{e_1,\cdots,e_{n-1}\}$ are tangent to $\p M^n$ and $e_n=\vec{n}$ is the outward normal vector. 

In the rest argument, we assume $x\in \partial M^n$ satisfying $|u(x)|< 1$. If $\partial M^n$ is smooth,  convex and $\frac{\partial u}{\partial \vec{n}}\big|_{\partial M^n}= 0$, then at $x$,
\begin{align}
\nabla _{\vec{n}}|\nabla u|^2&=\sum_{i=1}^n2\nabla _{\vec{n}}\nabla _{e_i}u \cdot \nabla _{e_i}u=\sum_{i=1}^n-2(\nabla _{e_i}\vec{n})u\cdot \nabla _{e_i}u, \nonumber \\
&=\sum_{1\le i,j\le n}-2 \langle \nabla _{e_i}\vec{n}, e_j  \rangle
 \nabla _{e_j}u \cdot \nabla _{e_i}u=-2\lan \nabla _{\nabla u}\vec{n}, \nabla u  \ran \le 0. \nonumber 
\end{align}
by the convexity of $\p M^n$. 

If $\partial M^n$ is smooth,  mean convex and $u\big|_{\partial M^n}= 0$, then note $\Delta u(x)= -\lambda u(x)= 0$, we have  
\begin{equation}\nonumber 
\begin{split}
\nabla _{\vec{n}}|\nabla u|^2&=\sum_{i=1}^n 2\nabla _{\vec{n}}\nabla _{e_i}u\cdot \nabla _{e_i}u=2\nabla _{\vec{n}}\nabla _{\vec{n}}u\cdot \nabla _{\vec{n}}u= 2\nabla ^2u(e_n,e_n)\cdot \nabla _{\vec{n}}u\\
&=2( \Delta u-\sum_{i=1}^{n-1} \nabla ^2u(e_i,e_i)        )\cdot \nabla _{\vec{n}}u=-2\sum_{i=1}^{n-1} \nabla ^2u(e_i,e_i)\cdot \nabla _{\vec{n}}u\\
&=-2\sum_{i=1}^{n-1} ( \nabla _{e_i}\nabla _{e_i}u-(\nabla _{e_i}e_i)u            )\cdot \nabla _{\vec{n}}u=2\sum_{i=1}^{n-1} \lan \nabla _{e_i}e_i, \nabla u  \ran\cdot \nabla _{\vec{n}}u\\
&=  2\sum_{i=1}^{n-1} \lan \nabla _{e_i}e_i, e_n  \ran\cdot (\nabla _{\vec{n}}u)^2\le 0,
\end{split}
\end{equation} 
where the last inequality comes from the mean-convex property of $\p M^n$. 

Now we have 
\begin{equation}
\nabla _{\vec{n}}|\nabla \sin^{-1}u|^2(x)=\frac{ (\nabla _{\vec{n}}|\nabla u|^2)(1-u^2)-|\nabla u|^2(-2u \nabla _{\vec{n}}u)  }{(1-u^2)^2}=\frac{\nabla _{\vec{n}}|\nabla u|^2}{1-u^2}\leq 0. \nonumber 
\end{equation}
\end{proof}

\begin{notation}\label{notation ellip operator}
{For $u\in C^\infty(M^n)$,  for any domain $\Omega\subseteq M^n$ with $|\nabla u|(x)\neq 0,  |u(x)|< 1$ in $\Omega$;  we define $\mathscr{L}_u: C^\infty(\Omega)\rightarrow C^\infty(\Omega)$ as follows:
\begin{align}
&\mathscr{L}_u(f)(x)\vcentcolon= \Delta f(x)+ \nabla f(x)\cdot \Psi(x)- \frac{2|\nabla u|^2}{(1- u^2)^2}f(x),  \quad \quad \forall f\in C^\infty(\Omega);  \nonumber \\
&\Psi(x)= \frac{\nabla u}{1- u^2}\Big\{-\frac{1}{2}\frac{(1- u^2)^2}{|\nabla u|^4} \nabla u\cdot \nabla \Big( \frac{|\nabla u|^2}{1- u^2}(x)\Big)- 2u\Big\} . \label{formula for Psi}
\end{align}
}
\end{notation}

The following lemma is the key differential inequality for make-up function $F(x)= \frac{|\nabla u|^2}{1- u^2}(x)- \lambda$, which will be crucially used in the proof Theorem \ref{thm splitting for Neumann with boundary}. 
\begin{lemma}\label{lem Laplace of gradient of angles}
{Suppose $M^n$  is a compact Riemannian manifold with $Rc\ge 0$, let $u$ be an eigenfunction with $\Delta u=-\lambda u$ in $M^n $, where $\lambda>0$. For any $x\in M^n $ with $|u(x)|< 1$ and $|\nabla u(x)|\neq 0$, let $F(x)= \frac{|\nabla u|^2}{1- u^2}(x)- \lambda$, we have
\begin{align}
\mathscr{L}_u(F)(x)\geq 0, \nonumber 
\end{align}
}
\end{lemma}

\pf
{Let $\{e_1,\cdots,e_n\} $ be an orthonormal basis at $T_{x}M$ with $\nabla _{e_i}e_j(x)=0$, $1\le i,j\le n$, where $e_1=\frac{\nabla u}{|\nabla u|}(x)$. Denote $u_i(x)=\nabla _{e_i}u(x)$ and $u_{ij}=\nabla _{e_j}\nabla _{e_i}u(x)$, then 
 $u_1(x)=|\nabla u|(x)$ and $u_i(x)=0$, where $2\le i\le n$. 

Direct computation yields
\begin{align}
\nabla F= 2|\nabla u| \Big(\frac{(\nabla ^2u) e_1}{1- u^2}+ \frac{u |\nabla u|^2  e_1}{(1- u^2)^2}\Big). \nonumber 
\end{align}

Multiplying both sides by $e_1$, then  we have
\begin{align}
u_{11}(x)= \frac{(1- u^2) (\nabla F) e_1}{2|\nabla u|} -\frac{|\nabla u|^2  u}{(1- u^2)}. \nonumber 
\end{align} 

And we get
\begin{align}
|\nabla ^2u|^2\geq u_{11}^2= \frac{|\nabla u|^4 u^2}{(1- u^2)^2}+ \frac{(1- u^2)\nabla F\cdot e_1}{2|\nabla u|} \Big\{ \frac{(1- u^2) \nabla F\cdot \nabla u}{2|\nabla u|^2}-2\frac{|\nabla u|^2 u}{1- u^2}\Big\} . \label{ineq of Hessian}
\end{align}


Direct computation reveals that
\begin{equation}
\Delta F= \Delta \frac{|\nabla u|^2}{1-u^2}=\frac{\Delta |\nabla u|^2}{1-u^2}-\frac{|\nabla u|^2\Delta (1-u^2)}{(1-u^2)^2}-2\nabla \frac{|\nabla u|^2}{1-u^2}\cdot \nabla \ln(1-u^2).\label{the laplace of F,10.18}
\end{equation} 

By the Bochner's formula, we know 
\begin{equation}
\label{the use about bochner,10.18}
\Delta |\nabla u|^2\ge 2|\nabla^2u|^2-2\lambda |\nabla u|^2.
\end{equation}

Plugging (\ref{the use about bochner,10.18}) and (\ref{ineq of Hessian}) into (\ref{the laplace of F,10.18}), we have
\begin{align}
\Delta F&\geq -2\nabla F\cdot \nabla [\ln(1- u^2)]+ \frac{2 \Big\{|\nabla u|^2(|\nabla u|^2- \lambda u^2)+ (1- u^2) (|\nabla ^2 u|^2- \lambda|\nabla u|^2)\Big\}}{(1- u^2)^2}\cdot \nonumber \\
&\ge \frac{2|\nabla u|^2}{(1- u^2)^2}F- \nabla F\cdot \Psi(x),
\end{align}
where $\Psi$ is defined as in (\ref{formula for Psi}).  Then the conclusion follows.
}
\qed

Now we are ready to show the sharp gradient estimate for (Neumann or Dirichlet) eigenfunctions in a uniform way. 

\begin{theorem}\label{thm LY estimate on mflds for neumann with boundary, 4.22}
For compact Riemannian manifold $(M^n, g)$ with $Rc\geq 0$, assume $\Delta u= -\lambda u, \displaystyle \max_{x\in M^n}|u(x)|= 1$, where $\lambda> 0$. If $\partial M^n$ is smooth,  convex and $\frac{\partial u}{\partial \vec{n}}\big|_{\partial M^n}= 0$, or if $\partial M^n$ is smooth,  mean convex and $u\big|_{\partial M^n}= 0$; then
\begin{align}
|\nabla u|^2+\lambda u^2\le \lambda. \nonumber 
\end{align}
\end{theorem}

\begin{remark}
{The Neumann case is originally proved in \cite{LY}, and the Dirichlet case is proved in \cite[Theorem $3.2$]{Yang}.
}
\end{remark}

\begin{proof}
\textbf{Step (1)}. Let  $u_{\ep}=\frac{u}{1+\ep}$, where  $\ep>0$, then $|u_{\ep}|<1$. We define $F_{\ep}(x)=\frac{|\nabla u_{\ep}|^2}{1-u_{\ep}^2}- \lambda$. Suppose $F_{\ep}(x_0)=\max\limits_{M^n}  \frac{|\nabla u_{\ep}|^2}{1-u_{\ep}^2}$. There are two cases :

If $x_0 \in M^n\backslash \p M^n$,  then we get $\nabla F_{\e}(x_0)=0$.

If $x_0 \in  \p M^n$, we can choose an orthonormal basis $\{e_1,\cdots, e_n \}$ in the neighborhood $U$ of $x_0$ such that $e_n(x)=\vec{n}(x)$ for all $x\in \p M^n\cap U$ and $\nabla _{e_n}e_i(x)=0$ for all $x\in U$, $1\le i\le n$.

From the definition of $x_0$, we know 
\begin{equation}
\label{neumann , ge 0, 4.22}
\nabla _iF_{\ep}(x_0)=0, \, i=1,\cdots,n-1 \,\, \text{and} \, \nabla _{\vec{n}}F_{\e}(x_0)\ge 0.
\end{equation} 

Note $\nabla _{\vec{n}}F_{\e}= \nabla _{\vec{n}}|\nabla \sin^{-1}u_\epsilon|^2$, from Lemma \ref{lem boundary computation,dirichlet,1.22}, we get that  $\nabla _{\vec{n}} F_{\e}(x_0)\leq 0$. Now $\nabla _{\vec{n}}F_{\e}(x_0)=0$ by (\ref{neumann , ge 0, 4.22}). Therefore, $\nabla F_{\e}(x_0)=0$ for $x_0\in \p M^n $.

Then  we obtain 
\begin{equation}
\nabla F_{\ep}(x_0)=0. \label{1.3.25}
\end{equation}

\textbf{Step (2)}. Then from the choice of $x_0$, for any $v\in ST_{x_0}M^n$ and $\gamma_{x_0}(tv)= \exp_{x_0}(tv)$, we get
\begin{align}
\nabla ^2_{v, v}F_{\ep}(x_0)&= 2\lim_{t\rightarrow 0}\frac{F_{\ep}(\gamma_{x_0}(tv))- F_{\ep}(x_0)- t\nabla _vF_{\ep}(x_0)}{t^2}= 2\lim_{t\rightarrow 0}\frac{F_{\ep}(\gamma_{x_0}(tv))- F_{\ep}(x_0)}{t^2} \nonumber \\
&\leq 0. \label{2.3.25}
\end{align}

We suppose $\nabla u_{\ep}(x_0)\neq 0$ in the rest argument, otherwise $|\nabla u_{\ep}|\equiv 0$ on $ M^n $, and  the conclusion is trivial.

Note $u_\epsilon$ is an eigenfunction with $|u_\epsilon|< 1$. From Lemma \ref{lem Laplace of gradient of angles}, we get 
\begin{align}
\mathscr{L}_{u_\epsilon}F_\epsilon(x_0)\geq 0.  \nonumber 
\end{align}

From the above, (\ref{1.3.25}) and (\ref{2.3.25}), we get $F_{\ep}(x_0)\leq 0$, which implies $|\nabla u_{\ep}|^2(x_0)+\lambda u_{\ep}^2(x_0)\le \lambda$. By the definition of $x_0$, we know $|\nabla u_{\ep}|^2+\lambda u_{\ep}^2(x)\le \lambda$ for all $x\in M^n$.
  Let $\ep \to 0$, we get $|\nabla u|^2+\lambda u^2\le \lambda$.
\end{proof}

\section{The local splitting induced by eigenfunctions}\label{sec local splitting}

\begin{definition}\label{def Mp definition}
{For $u\in C^\infty(M^n)$ and $p\in M^n$, if $\nabla u(p)\neq 0$, let $M_{u, p}$ be the connected component of $M^n \backslash  \mathscr{C}(u)$ containing $p$, where $\mathscr{C}(u)=\{x\in M^n: |\nabla u|(x)=0 \}$. When the context is clear, we also use $M_p$ for simplicity.
}
\end{definition}

\begin{prop}\label{prop the isometric map,4.23}
For compact Riemannian manifold $(M^n, g)$ with $Rc\geq 0$. Assume $\Delta u= -\lambda u, \displaystyle \max_{x\in M^n}|u(x)|= 1$, where $\lambda> 0$. If $|\nabla \sin^{-1}u|^2(p)= \lambda$ holds at some point $p$ with $|u(p)|< 1$ , then  $\displaystyle N^{n-1}\vcentcolon= u^{-1}(u(p))\cap M_p$ is connected and $\partial N^{n- 1}= \emptyset$.  Furthermore, 
\begin{enumerate}
\item[(N)]. If $\partial M^n= \emptyset$,  then $M_p$ is isometric to $\displaystyle N^{n-1}\times \left(-\frac{\pi}{2\sqrt{\lambda}},  \frac{\pi}{2\sqrt{\lambda}} \right)$ and $\partial N^{n- 1}= \emptyset$. 
\item[(NB)]. If $\partial M^n\neq \emptyset$ is smooth,  convex and $\frac{\partial u}{\partial \vec{n}}\big|_{\partial M^n}= 0$,  then $M_p$ is isometric to $\displaystyle N^{n-1}\times \left(-\frac{\pi}{2\sqrt{\lambda}},  \frac{\pi}{2\sqrt{\lambda}} \right)$.
\item[(DB)]. If $\partial M^n\neq \emptyset$ is smooth,  mean convex and $u\big|_{\partial M^n}= 0$; then $\partial N^{n- 1}= \emptyset$. And
\begin{enumerate}
\item[(DB-a)] If $M_p\cap \partial M^n= \emptyset$,  then $M_p$ is isometric to $\displaystyle N^{n-1}\times \left(-\frac{\pi}{2\sqrt{\lambda}},  \frac{\pi}{2\sqrt{\lambda}} \right)$.  
\item[(DB-b)] If $M_p\cap \partial M^n\neq \emptyset$,  then $M_p$ is isometric to $\displaystyle N^{n-1}\times \left(-\frac{\pi}{2\sqrt{\lambda}},  0\right]$.  
\end{enumerate}
\end{enumerate}
\end{prop}

\pf
{\textbf{Step (1)}. First note that for any $x\in M_p$, we have $|\nabla u|(x)\neq 0$ and $|u|(x)<1$ by the definition of $M_p$ and Theorem \ref{thm LY estimate on mflds for neumann with boundary, 4.22}.

Let  $V=\{x\in M_p\,\,|\,\, |\nabla  \sin^{-1}u|(x)=\sqrt{\lambda}    \}$, which is nonempty since $p\in V$. It's evident  that $V$ is  a closed set in $M_p$. 

By Lemma \ref{lem Laplace of gradient of angles}, we have 
\begin{align}
\mathscr{L}_u(F)(x)\geq 0, \quad \quad \quad \forall x\in M_p; \nonumber 
\end{align}
where $F(x)= |\nabla \sin^{-1}u|^2(x)- \lambda$.

Given $x_1\in V $, note $F$ attains its maximum at $x_1$ with $F(x_1)=0$ by Theorem \ref{thm LY estimate on mflds for neumann with boundary, 4.22}. Note $\frac{\partial F}{\partial\vec{n}}\big|_{\partial M^n}\leq 0$ by Lemma \ref{lem boundary computation,dirichlet,1.22}, we apply the Strong Maximum Principle (see \cite[Lemma $3.4$ and Theorem $3.5$]{GT}) to get   
$F\big|_{B_{x_1}(r)\cap M_p}\equiv 0$. Therefore $(B_{x_1}(r)\cap M_p)\subset V$, and $V$ is open in $M_p$.

Because $M_p$ is connected , we conclude that  $|\nabla \sin^{-1}u|(x)=\sqrt{\lambda}$ for any $x\in M_p$.

Let $\Theta(x)= \frac{1}{\sqrt{\lambda}}\sin^{-1}u(x)$ for $x\in M_p$, then direct computation combining $|\nabla  \Theta|(x)=1, \Delta u= -\lambda u$ yields
\begin{align}
\Delta \Theta= 0, \quad \quad \quad \quad \forall x\in M_p. \nonumber 
\end{align}

From the Bochner formula and $Rc\ge 0$, we can infer that 
\begin{equation*}
\begin{split}
|\nabla ^2\Theta|^2&=\frac{1}{2}\Delta|\nabla \Theta|^2-\langle \nabla  \Delta \Theta, \nabla \Theta  \rangle- Rc(\nabla \Theta,\nabla \Theta) \le 0.
\end{split}
\end{equation*}

Hence $|\nabla ^2\Theta|\equiv 0$ and $|\nabla \Theta|\equiv 1$ in $M_p$. 

\textbf{Step (2)}.  Define $\mathcal{L}_t= \big\{\ x\in M_p: \Theta(x)= t\big\}$, from $|\nabla \Theta|\equiv 1$ and the Implicit Function Theorem, we know that $\mathcal{L}_t$ is $(n- 1)$-dim submanifold of $M^n$ if $\mathcal{L}_t\neq \emptyset$. 

For any $x\in M_p$, let $\gamma_x$ be the gradient curve of $\nabla \Theta$ with $\gamma_x(0)= x$ in $M_p$. Define 
\begin{align}
T_x\vcentcolon= \sup\{t\geq 0: \gamma_x(t)\in M_p\}, \quad \quad \hat{T}_x\vcentcolon= \sup\{t\geq 0: \gamma_x(-t)\in M_p\}, \quad \quad \forall x\in M_p.  \nonumber 
\end{align}

Now we prove the following claim:
\begin{claim}\label{claim proj map exists}
{For any $x\in M_p$, the set $\gamma_x\cap \mathcal{L}_0$ contains a unique point, which we denote by $\mathcal{P}(x)$.  
}
\end{claim}


\pf[Proof of Claim \ref{claim proj map exists}]
{If $\Theta(x)= 0$,  then $\mathcal{P}(x)= x$.


Now we consider the case $\Theta(x)< 0$. We firstly show $T_x\geq -\Theta(x)$. By contradiction, assume $T_x< -\Theta(x)$.  Then by $|\nabla \Theta|\leq 1$, we have
\begin{align}
-\frac{\pi}{2\sqrt{\lambda}}< \Theta(x)\leq \Theta(\gamma_x(t))\leq \Theta(x)+ t\leq \Theta(x)+ T_x< 0, \quad \quad \forall t\in [0, T_x).\nonumber 
\end{align}
Using $|\nabla \Theta|\equiv 1$ in $M_p$ and $u(z)= \sin(\sqrt{\lambda}\Theta(z))$ for any $z\in M_p$,  we get
\begin{align}
|\nabla u|(\gamma_x(t))\geq \sqrt{\lambda}\cdot \sqrt{1- \sin^2(\sqrt{\lambda}\Theta(x))}> 0, \quad \forall t\in [0, T_x).\label{Du not vanish}
\end{align}

From (\ref{Du not vanish}), we get 
\begin{align}
\lim_{t\rightarrow T_x}|\nabla u|(\gamma_x(t))\geq \sqrt{\lambda}\cdot \sqrt{1- \sin^2(\sqrt{\lambda}\Theta(x))}> 0, \label{Du gamma Tx is not vanishing}
\end{align}
which implies $\gamma_x(T_x)$ is well-defined and $\gamma_x(T_x)\in M_p$. 

Furthermore, 
\begin{align}
u(\gamma_x(t))= \sin(\sqrt{\lambda}\Theta(\gamma_x(t)))\leq \sin(\sqrt{\lambda}[\Theta(x)+ T_x])< 0, \quad \quad \forall t\in [0, T_x).\label{u is not zero} 
\end{align}

For case (DB), from (\ref{u is not zero}), we know $\gamma_x(T_x)\notin \partial M^n $. Then  there is $\delta> 0$ such that $\gamma_x(T_x+ \delta)\in M_p$.

For case (NB), from (\ref{Du gamma Tx is not vanishing}), we know $|\nabla u|\left(\gamma_x(T_x)\right)>0$. 
\begin{enumerate}
\item[(a)]. If  $\gamma_x(T_x)\notin \partial M^n $, there is $\delta> 0$ such that $\gamma_x(T_x+ \delta)\in M_p$. 
\item[(b)]. If $\gamma_x(T_x)\in \partial M^n $, from $\gamma_x'(T_x)=\nabla \Theta \left(\gamma_x(T_x)\right)=\frac{\nabla u}{| \nabla u|}\left(\gamma_x(T_x)\right)\neq 0$ and $\frac{\p u}{\p \vec{n}}\left(\gamma_x(T_x)\right)=0$; we obtain $\gamma_x'(T_x)\in T_{\gamma_x(T_x)}\p M^n$. 

Combining $\p M^n$ is smooth and $(\nabla u\cdot \vec{n})\big|_{\partial M^n}= 0$, we can extend $\gamma_x$ (which is the gradient curve of $u$) in $\partial M^n$. Hence there is $\delta> 0$ such that $\gamma_x(T_x+ \delta)\in M_p$. 
\end{enumerate}

By the above argument, there is $\delta> 0$ such that $\gamma_x(T_x+ \delta)\in M_p$. It contradicts the definition of $T_x$. Therefore $T_x\geq -\Theta(x)$. 

Next we prove that $\gamma_x(-\Theta(x))\in M_p$. Like (\ref{Du not vanish}), we have 
\begin{align}
|\nabla u|(\gamma_x(t))\geq \sqrt{\lambda}\cdot \sqrt{1- \sin^2(\sqrt{\lambda}\Theta(x))}> 0, \quad \quad \forall t\in [0, -\Theta(x)).\label{Du not vanish on bigger interval}
\end{align}
This implies $\gamma_x(-\Theta(x))\in M_p$ is well-defined. 

Now $\mathcal{P}(x)= \gamma_x(-\Theta(x))$ is well defined. 

If $\Theta(x)> 0$, consider the gradient curve of $-\nabla \Theta$ as above, we can get that $\mathcal{P}(x)$ is also well-defined. By all the above, we know that $\mathcal{P}(x)$ is well-defined for $x\in M_p$. 
}
\qed

\textbf{Step (3)}.  Now for cases (N), (NB) and (DB-a),  we will show that for each $x\in \mathcal{L}_0$,  there is 
\begin{align}
T_x= \hat{T}_x= \frac{\pi}{2\sqrt{\lambda}}.  \label{end pts of interval}
\end{align}

By contradiction,  if $T_x< \frac{\pi}{2\sqrt{\lambda}}$,  then $\Theta(\gamma_x(T_x))< \frac{\pi}{2\sqrt{\lambda}}$.

  We get $|\nabla u|(\gamma_x(T_x))= \sqrt{\lambda}\cdot \sqrt{1- u^2(\gamma_x(T_x))}> 0$.  Hence there is $\delta> 0$ such that $\gamma_x(T_x+ \delta)\in M_p$,  which contradicts the definition of $T_x$.

  Similar argument applies on $\hat{T}_x$. Therefore (\ref{end pts of interval}) is proved.

Similarly,  for case (DB-b) we can show
\begin{align}
&T_x\cdot \hat{T}_x= 0,  \quad \text{and}\quad  T_x+ \hat{T}_x= \frac{\pi}{2\sqrt{\lambda}},  \quad \quad \quad \forall x\in \mathcal{L}_0\cap \partial M^n .\label{end pts of interval-DB-b}
\end{align}

For cases (N),  (NB) and (DB-a), we define the map 
\begin{align}
\psi(x)= (\Theta(x), \mathcal{P}(x)): M_p\rightarrow (-\frac{\pi}{2\sqrt{\lambda}}, \frac{\pi}{2\sqrt{\lambda}})\times \mathcal{L}_0. \label{def of psi for three cases} 
\end{align}
Because $x= \gamma_{\mathcal{P}(x)}(\Theta(x))$ and (\ref{end pts of interval}), we get that $\psi$ is a bijective map. Hence $\psi$ is a diffeomorphism between $M_p$ and $(-\frac{\pi}{2\sqrt{\lambda}}, \frac{\pi}{2\sqrt{\lambda}})\times \mathcal{L}_0$. 

For case (DB-b), we get that the map $\psi(x)= (\Theta(x), \mathcal{P}(x))$ is a diffeomorphism between $M_p$ and $(-\frac{\pi}{2\sqrt{\lambda}}, 0]\times \mathcal{L}_0$ (or $[0, \frac{\pi}{2\sqrt{\lambda}})\times \mathcal{L}_0$).

\textbf{Step (4)}. Assume $\xi_1'(0)= v_1\in T_xM^n$, now using $|\nabla^2\Theta|\equiv 0$ and the definition of $\mathcal{P}$, we get
\begin{align}
\langle\psi_*v_1, \psi_*v_1\rangle&= \frac{d}{dt}\psi(\xi_1)\cdot \frac{d}{dt}\psi(\xi_1)= \langle \frac{d}{dt}(\Theta(\xi_1(t)), \mathcal{P}(\xi_1(t))), \frac{d}{dt}(\Theta(\xi_1(t)), \mathcal{P}(\xi_1(t)))\rangle \nonumber \\
&= \langle \frac{d}{dt}\Theta(\xi_1(t)), \frac{d}{dt}\Theta(\xi_1(t))\rangle+ \langle \frac{d}{dt}\mathcal{P}(\xi_1(t)), \frac{d}{dt}\mathcal{P}(\xi_1(t))\rangle \nonumber \\
&= \langle \nabla\Theta\cdot v_1, \nabla\Theta\cdot v_1\rangle+ \langle \nabla\mathcal{P}\cdot v_1, \nabla\mathcal{P}\cdot v_1\rangle= \langle v_1, v_1\rangle. \nonumber 
\end{align}
Hence $\psi$ is a local isometry, the conclusion $\psi$ is an  isometry follows from all the above.

Therefore $N^{n-1}= u^{-1}(u(p))\cap M_p$ is isometric to $\mathcal{L}_0$. Because $M_p$ is connected, we get $N^{n- 1}$ is connected, by the above isometry between $M_p$ and the product of $N^{n- 1}$ with an interval. 

Also we obtain $\partial N^{n-1}= \emptyset$ by the above isometry and $\partial M^n$ is smooth.
}
\qed

\section{The global splitting}\label{sec splitting extension}

\begin{theorem}\label{thm splitting for Neumann with boundary}
{For compact Riemannian manifold $(M^n, g)$ with $Rc\geq 0$ and smooth (possibly empty) boundary . Assume $\Delta u= -\lambda u, \displaystyle \max_{x\in M^n}|u(x)|= 1$, where $\lambda> 0$. If $|\nabla \sin^{-1}u|^2(p)= \lambda$ holds at some point $p$, where $|u(p)|< 1$. Define $\mathcal{N}(u)= \#(M^n- u^{-1}(0))$ as the number of connected components of $M^n- u^{-1}(0)$ (the number of  nodal domains of $u$); then:
\begin{enumerate}
\item[(1)]. If $\partial M^n= \emptyset$, then $M^n$ is isometric to $\mathbb{S}^1(\frac{\mathcal{N}(u)}{2\sqrt{\lambda}})\times N^{n- 1}$, where $\mathbb{S}^1(\frac{\mathcal{N}(u)}{2\sqrt{\lambda}})$ is a round circle with radius $\frac{\mathcal{N}(u)}{2\sqrt{\lambda}}$ and $\mathcal{N}(u)= 2k$ for some $k\in \mathbb{Z}^+$;
\item[(2)]. If $\partial M^n\neq \emptyset$ is convex and $\frac{\partial u}{\partial \vec{n}}\big|_{\partial M^n}= 0$, then $M^n$ is isometric to $[0, \frac{\mathcal{N}(u)- 1}{\sqrt{\lambda}}\cdot \pi]\times N^{n- 1}$ and $\mathcal{N}(u)\geq 2$;
\item[(3)]. If $\partial M^n\neq \emptyset$ is mean convex and $u\big|_{\partial M^n}= 0$; then $M^n$ is isometric to $[0, \frac{\mathcal{N}(u)}{\sqrt{\lambda}}\cdot \pi]\times N^{n- 1}$ and $\mathcal{N}(u)\geq 1$;
\end{enumerate}
where $N^{n- 1}$ is a compact Riemannian manifold with $Rc\geq 0$ and $\partial N^{n- 1}= \emptyset$. }
\end{theorem}

\pf
{\textbf{Step (1)}.  Define $N^{n- 1}= u^{-1}(-1)\cap \overline{M_p}$, which is an $(n-1)$-dim hypersurface in $M^n$ by Proposition \ref{prop the isometric map,4.23} (if $u^{-1}(-1)\cap \overline{M_p}= \emptyset$,  we define $N^{n- 1}= u^{-1}(1)\cap \overline{M_p}$, similar argument applies);  and $U_\epsilon(N^{n- 1})= \{x\in M^n: d(x, N^{n- 1})< \epsilon\}$. Then we have a coordinate chart $\{x_1, x'\}$ for $U_{r_0}(N^{n- 1})$ locally,  where $x'=(x_2,\cdots,x_n)$ is a coordinate chart for $N^{n- 1}$ and 
\begin{align}
x_1\big|_{M_p\cap U_{r_0}(N^{n- 1})}= d(x, N^{n- 1}), \quad \quad \quad x_1\big|_{U_{r_0}(N^{n- 1})- M_p}= -d(x, N^{n- 1}), \nonumber 
\end{align}
where $r_0> 0$ is to be determined later. We define $U_{r_0}^{-}(N^{n- 1})\vcentcolon= U_{r_0}(N^{n- 1})\cap \{x\in M^n: x_1< 0\}$.

Consider the metric $g$ in $U_{r_0}(N^{n- 1})$,  then we can write  $\displaystyle g=dx_1^2+ \sum_{2\leq i, j\leq n}f_{ij}(x_1,x')dx_idx_j$  locally.  By Proposition \ref{prop the isometric map,4.23},  there is $a> 0$,  such that the coefficients $f_{ij}$  satisfy:
$$
f_{ij}(x_1,x')=\begin{cases}
&f_{ij}(0,x'),\,\, 0\le x_1\le a,\\
&f_{ij}(x_1,x'),\,\, -a\le x_1\le 0.
\end{cases}
$$  

For any $q\in N^{n-1}$,  without loss of generality we can assume  $q=(0,0,\cdots,0)$.  Now we get 
\begin{align}
f_{ij}(0,x')=\delta_{ij}+O(|x'|).  \nonumber 
\end{align}

Denote $G=\mathrm{det}(f_{ij})_{2\le i,j\le n}$,  we get 
$$(\log \sqrt{G})_{x_1}=\frac{1}{2}(\log G)_{x_1}=O(x_1). $$ 

We use $u_i, u_{ij}$ to denote $\frac{\partial u}{\partial x_i},  \frac{\partial^2 u}{\partial x_i\partial x_j}$ in the rest argument,  similar notation applies for other functions.  

Note $u_{11}(q)= \lambda$ because $u\big|_{M_p}(x_1, x')= \sin(\sqrt{\lambda}x_1- \frac{\pi}{2})$ from Proposition \ref{prop the isometric map,4.23}.  Also we have 
\begin{align}
&u_{\alpha}(q)= 0, \quad \quad \text{if}  \quad \quad \alpha= (\alpha_1, \cdots, \alpha_k)\neq (1, \cdots, 1), \nonumber \\
&u_{\beta}(q)= 0, \quad \quad \text{if}  \quad \quad \beta= (1, \cdots, 1) \ \text{and there are odds number of $1$}, \nonumber 
\end{align}

Note $N^{n- 1}$ is compact, there is a finite coordinate chart covering of $N^{n- 1}$.  Using the Taylor's expansion of $u(x_1,x')$ at $q$ about $x_1$,  where $q$ can be freely chosen from $N^{n-1}$; we get 
\begin{align}
u(x)= -1+ \frac{1}{2}\lambda x_1^2+ O(|x_1|^4), \quad \quad \quad \forall x\in U_{r_0}(N^{n- 1}). \label{Taylor of u in nbhd of N}
\end{align}

By (\ref{Taylor of u in nbhd of N}), we can choose $r_0> 0$ such that
\begin{align}
|\nabla u \big|_{\overline{U_{r_0}(N^{n- 1})}- N^{n- 1}} |>0,  \quad \quad \text{and} \quad \quad |u\big|_{\overline{U_{r_0}(N^{n- 1})}- N^{n- 1}}|< 1. \nonumber 
\end{align}

\textbf{Step (2)}.  Now we define $F: \overline{U_{r_0}^-(N^{n- 1})}\rightarrow \mathbb{R}$ as follows:
\begin{equation}\nonumber 
		F(x)\vcentcolon= \left\{
		\begin{array}{rl}
			&\frac{|\nabla u|^2}{1- u^2}(x)- \lambda \ ,  \quad \quad \quad \quad\quad  \quad\quad \quad x\in \overline{U_{r_0}^-(N^{n- 1})}- N^{n- 1}, \\
			&0\ ,  \quad \quad \quad \quad \quad\quad \quad x\in N^{n- 1}. 
		\end{array} \right.
	\end{equation}
From the Taylor's expansion of $u$ at any point of $N^{n- 1}$,  we get that $F$ is continuous on $\overline{U_{r_0}^-(N^{n- 1})}$.  

Then from (\ref{Taylor of u in nbhd of N}), we have
\begin{align}
\frac{-2|\nabla u|^2}{(1- u^2)^2}&= -\frac{2}{x_1^2}+ O(1). \label{zero order est}
\end{align}

Direct computation yields 
\begin{align}
\Psi(x)&= \frac{\nabla u}{1- u^2}\cdot \{-\frac{1}{2}\frac{(1- u^2)^2}{|\nabla u|^4} \nabla u\cdot \nabla \Big( \frac{|\nabla u|^2}{1- u^2}(x)\Big)- 2u\} \nonumber \\
&= \Big(\frac{2}{ x_1}+ O(x_1)\Big)\cdot \partial x_1. \label{est of Psi in r}
\end{align}

 Consider $v(x)= x_1^2- x_1$, then $v$ satisfies 
\begin{align}
v\big|_{U_{r_0}^-(N^{n- 1})}\geq 0, \quad \quad v\big|_{x_1= 0}= 0, \quad \quad \frac{\partial}{\partial x_1}v\big|_{x_1= 0}= -1< 0. \label{C0 assumption of v} 
\end{align}

Furthermore, using (\ref{est of Psi in r}) and (\ref{zero order est}), we have 
\begin{align}
\mathscr{L}_u(v)&=  \Delta v(x)+ \nabla v(x)\cdot \Psi(x)- \frac{2|\nabla u|^2}{(1- u^2)^2}v(x)\nonumber \\
&= v_{11}+(\log \sqrt{G})_{x_1}v_{1}+\sum_{2\le i,j\le n} f^{ij}v_{ij}+\sum_{2\le i,j\le n} \{\frac{\p f^{ij}}{\p x_i}+ (\log \sqrt{G})_{x_i}f^{ij}  \}v_{j} \nonumber \\
&+ v_{1} (\frac{2}{x_1}+O(x_1) )-\frac{2}{x_1^2}v+O(1)v \nonumber \\
&=2+O(x_1)+ (2x_1- 1) (\frac{2}{x_1}+O(x_1) )-\frac{2}{x_1^2} \{ x_1^2-x_1\}   \nonumber \\
&=4+O(x_1)>0 , \nonumber
\end{align}
in the last inequality above we use $x_1\leq r_0<< 1$. 

Now there is $r_0> 0$ and $v\in C^\infty(U_{r_0}(N^{n- 1}))$ satisfying
\begin{align}
\mathscr{L}_u(v)(x)> 0,   \quad \quad \quad \forall x\in U_{r_0}^-(N^{n- 1}). \label{L of v}
\end{align}

\textbf{Step (3)}.  For any $q\in N^{n- 1}$,  let $\vec{n}= \partial x_1$ be the outer normal unit vector of $\partial U_{r_0}^-(N^{n- 1})$ at $q$. Using (\ref{Taylor of u in nbhd of N}) we can compute directly that 
\begin{align}
\frac{\partial F}{\partial\vec{n}}(q)= 0. \label{normal deri of F is 0}
\end{align}

We will show that there is $z\in U_{r_0}^-(N^{n- 1})$ such that $F(z)= F(q)= 0$. By contradiction, otherwise
\begin{align}
F(z)< F(q)= 0, \quad \quad \quad \forall z\in U_{r_0}^-(N^{n- 1}). \label{strict smaller inside}
\end{align}

From Lemma \ref{lem Laplace of gradient of angles}, we have
\begin{align}
\mathscr{L}_u(F)\geq 0. \label{subharmonic condition-F}
\end{align}

Now we consider $F+ \epsilon\cdot v$. Because of (\ref{strict smaller inside}), we can choose $\epsilon> 0$ such that 
\begin{align}
(F+ \epsilon  v)\big|_{x_1= -r_0}\leq 0. \label{one boundary condition}
\end{align}

Then we have
\begin{align}
(F+ \epsilon v)\big|_{\partial U_{r_0}^{-}(N^{n- 1})}\leq 0. \nonumber 
\end{align}

By (\ref{L of v}),  we get $\mathscr{L}_u(F+ \epsilon\cdot v)\big|_{U_{r_0}^{-}(N^{n- 1})} > 0$,  then from the Maximum principle,  we get
\begin{align}
\sup_{x\in U_{r_0}^{-}(N^{n- 1})} (F+ \epsilon v)(x)\leq \max\{0, \sup_{x\in \partial U_{r_0}^{-}(N^{n- 1})}(F+ \epsilon v)(x)\}= 0. \nonumber  
\end{align}

Note $(F+ \epsilon v)(q)= 0$,  then we obtain
\begin{align}
\frac{\partial}{\partial x_1}\big|_{x= q}(F+ \epsilon v)(x)\geq 0. \nonumber
\end{align}

Combining (\ref{C0 assumption of v}), this yields
\begin{align}
\frac{\partial}{\partial x_1}\big|_{x= q}F(x)\geq -\epsilon \frac{\partial}{\partial x_1}\big|_{x= q}v(x)= \epsilon> 0. \nonumber  
\end{align}
This contradicts (\ref{normal deri of F is 0}). 

Therefore, we find $z\in U_{r_0}^-(N^{n- 1})$ such that $F(z)= F(q)= 0$, which implies $\displaystyle F(z)= \max_{x\in \overline{U_{r_0}^-(N^{n- 1})}}F(x)= 0$. 
Note $|\nabla u(z)|>0$, $|u(z)|<1$ and $|\nabla \sin^{-1}u(z)|^2=\lambda$;  from Proposition \ref{prop the isometric map,4.23}, we get the extended splitting domain on $M^n$. 

The main conclusion follows by the above and the induction on $\mathcal{N}(u)$.
}

\qed

\section*{Acknowledgments}
We thank Xiaodong Wang for his comments on the early version of this paper.

\end{document}